\documentclass{amsart}

\newtheorem{theorem}{Theorem}[section]
\newtheorem{lemma}[theorem]{Lemma}

\newtheorem{corollary}[theorem]{Corollary}
\newtheorem{proposition}[theorem]{Proposition}
\usepackage{graphicx,amssymb}

\theoremstyle{definition}

\newtheorem{example}[theorem]{Example}

\theoremstyle{remark}
\newtheorem{remark}[theorem]{Remark}

\numberwithin{equation}{section}

\begin{document}

\title{Left-orderability and exceptional Dehn surgery on twist knots}

\author{Masakazu Teragaito}
\address{Department of Mathematics and Mathematics Education, Hiroshima University,
1-1-1 Kagamiyama, Higashi-hiroshima, Japan 739-8524}
\email{teragai@hiroshima-u.ac.jp}
\thanks{
Partially supported by Japan Society for the Promotion of Science,
Grant-in-Aid for Scientific Research (C), 22540088.
}%

\subjclass[2010]{Primary 57M25; Secondary 06F15}



\keywords{left-ordering, twist knot, Dehn surgery}

\begin{abstract}
We show that any exceptional non-trivial Dehn surgery on a twist knot, except the trefoil,
yields a $3$-manifold whose fundamental group is left-orderable.
This is a generalization of a result of Clay, Lidman and Watson, and
also gives a new supporting evidence for a conjecture of Boyer, Gordon and Watson.
\end{abstract}

\maketitle

\section{Introduction}

A group is \textit{left-orderable\/} if it admits a strict total ordering
which is invariant under left-multiplication.
It is well known that any knot group or link group is left-orderable (see \cite{BRW}).
More generally, many classes of $3$-manifolds are known to have
left-orderable fundamental groups.

Boyer, Gordon and Watson \cite{BGW} state a conjecture that
an irreducible rational homology $3$-sphere is an $L$-space if and only if
its fundamental group is not left-orderable.
Here, an $L$-space is a rational homology sphere $M$ whose Heegaard Floer homology
$\widehat{HF}(M)$ is a free abelian group of rank equal to $|H_1(M)|$, introduced by 
Ozsv\'{a}th and Szab\'{o} \cite{OS}.
This conjecture is verified for Seifert fibered manifolds, Sol manifolds, and
double branched covers of non-split alternating links in \cite{BGW}.

On the other hand, if a knot admits Dehn surgery yielding an $L$-space, referred to as an $L$-space surgery,
then there are some constraints for the knot.
For example, its Alexander polynomial has a specified form \cite{OS}, and
such knot must be fibered \cite{N}.
Therefore, it is not going too far to say that most knots do not admit an $L$-space surgery.
Thus we can expect that any non-trivial Dehn surgery on a hyperbolic knot, which does not admit an $L$-space surgery, yields a $3$-manifold whose fundamental group is left-orderable.

In this direction, 
Boyer, Gordon and Watson \cite{BGW} show that if $K$ is the figure-eight knot and
$-4<r<4$, then $r$-surgery on $K$ yields a $3$-manifold with left-orderable fundamental group.
Furthermore, Clay, Lidman and Watson \cite{CLW} show
that this also holds for $r=\pm 4$.
(Note that the figure-eight knot does not admit an $L$-space surgery.)

In this paper, we will examine the $m$-twist knot in the $3$-sphere, illustrated in
Figure \ref{fig:twistknot}.
We adopt the convention that the horizontal twists are right-handed if $m$ is positive,
left-handed if $m$ is negative.
Thus, the $1$-twist knot is the figure-eight knot, and the $(-1)$-twist knot is the right-handed trefoil.
Also, if $|m|\ge 2$, then the $m$-twist knot is hyperbolic and non-fibered (see \cite{BZ}).

\begin{figure}[ht]
\includegraphics*[scale=0.6]{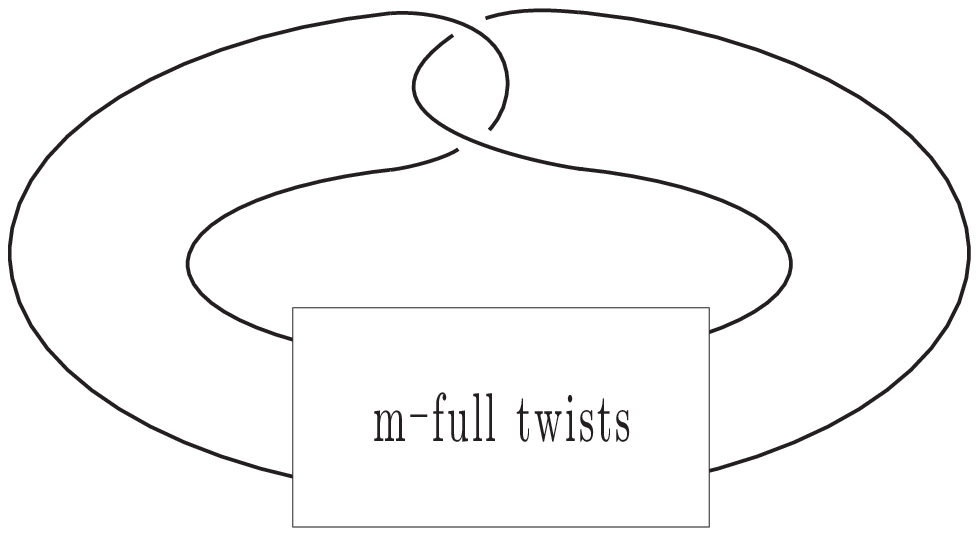}
\caption{}\label{fig:twistknot}
\end{figure}

The purpose of this paper is to verify that $4$-surgery on the $m$-twist knot with $|m|\ge 2$
yields a graph manifold whose fundamental group is left-orderable.
Since such a twist knot is non-fibered, it does not admits an $L$-space surgery.
(This fact also follows from the form of its Alexander polynomial.)
Thus the following theorem provides a new supporting evidence for
the conjecture of Boyer, Gordon and Watson mentioned above.

\begin{theorem}\label{thm:main}
Let $K$ be the $m$-twist knot with $|m|\ge 2$.
Then $4$-surgery on $K$ yields a graph manifold whose fundamental group is
left-orderable.
\end{theorem}

Our argument follows that of Clay, Lidman and Watson \cite[Section 4]{CLW} for the case of
$4$-surgery of the figure-eight knot.
They make use of the Dubrovina-Dubrobin ordering for the braid group $B_3$ of order $3$,
which is isomorphic to the knot group of the trefoil, but
we need some left-orderings for torus knot groups defined by Navas \cite{Na}.

By combining with known results,
we can immediately prove the following.

\begin{corollary}\label{cor}
Let $K$ be a hyperbolic twist knot.
Then any exceptional non-trivial Dehn surgery on $K$ yields
a $3$-manifold whose fundamental group is left-orderable.
\end{corollary}

The author would like to thank Tetsuya Ito for useful comments on left-orderings.

\section{Fundamental group}

Let $K$ be the $m$-twist knot.
We can assume that $m\ne 0,-1$.
It is well known that $4$-surgery on $K$ yields a toroidal manifold.
In fact, the manifold is a graph manifold.
In this section, we will examine the structure of the manifold
by using the Montesinos trick, and get a presentation of its fundamental group.

As shown in Figure \ref{fig:montesinos},
put $K$ in a symmetric position.
By taking the quotient under the involution map,
we have a tangle description of the knot exterior.
This means that the double branched cover of the (outside) ball branched over
the two strings recovers the knot exterior.

\begin{figure}[htbp]
\includegraphics*[scale=0.7]{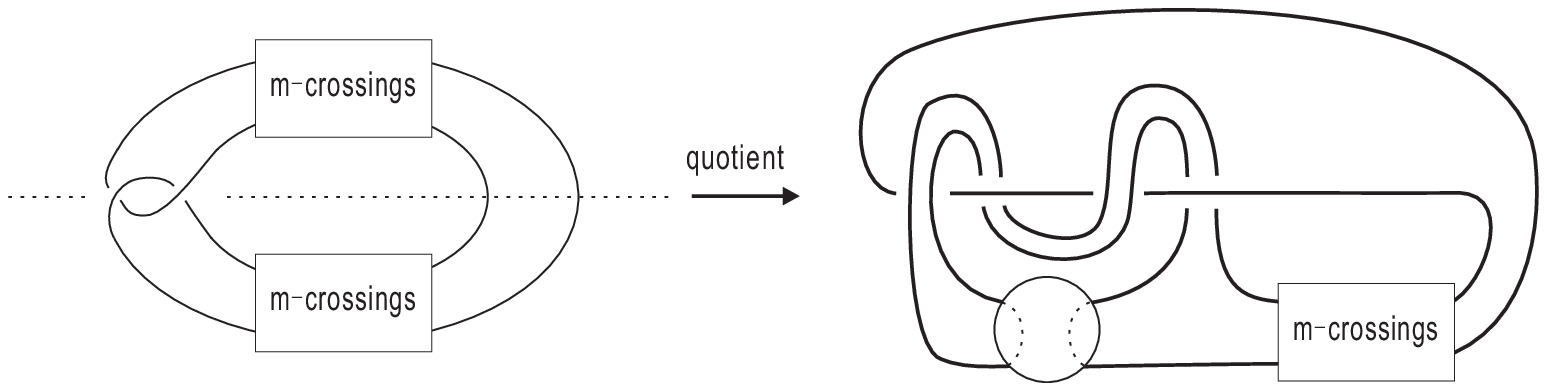}
\caption{}\label{fig:montesinos}
\end{figure}

If we fill the $\infty$-tangle, as indicated by dotted lines in Figure \ref{fig:montesinos}, into the inner ball,
then it gives an unknot.
Here, we choose the framing so that
the $0$-tangle filling corresponds to $4$-surgery.
Figure \ref{fig:montesinos2} shows the $0$-tangle filling yields a link
with a trivial component.
Let $S$ be the $2$-sphere illustrated there which gives
an essential tangle decomposition.

\begin{figure}[htbp]
\includegraphics*[scale=0.7]{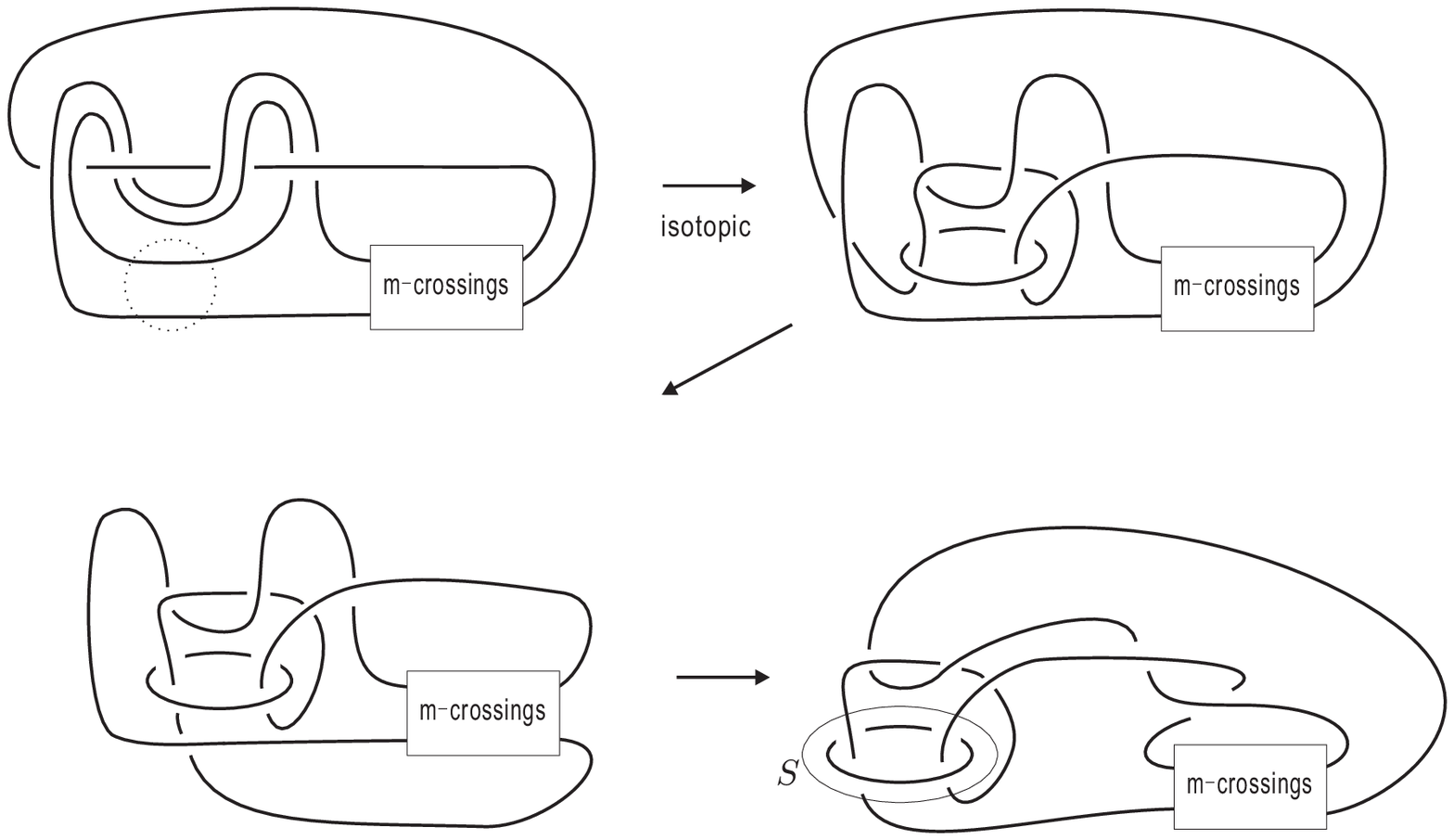}
\caption{}\label{fig:montesinos2}
\end{figure}

One side of $S$ is the Montesinos tangle $M(-1/2,-m/(2m+1))$, and the other
side is the Montesinos tangle $M(-1/2,1/2)$.
Then the double branched cover $M_1$ of $M(-1/2,-m/(2m+1))$ is the exterior of the torus knot of type $(2,2m+1)$, and the double cover $M_2$  of $M(-1/2,1/2)$ is the twisted $I$-bundle over the Klein
bottle.
Thus the resulting manifold $M$ of $4$-surgery on $K$ is $M_1\cup M_2$.

We have $\pi_1(M_1)=\langle a,b : a^2=b^{2m+1}\rangle$.
See Figure \ref{fig:torusknot}.
The meridian $\mu$ is $b^{-m}a$ and the regular fiber $h$ with respect to a (unique)
Seifert fibration, is $a^2\ (=b^{2m+1})$.

\begin{figure}[htbp]
\includegraphics*[scale=0.5]{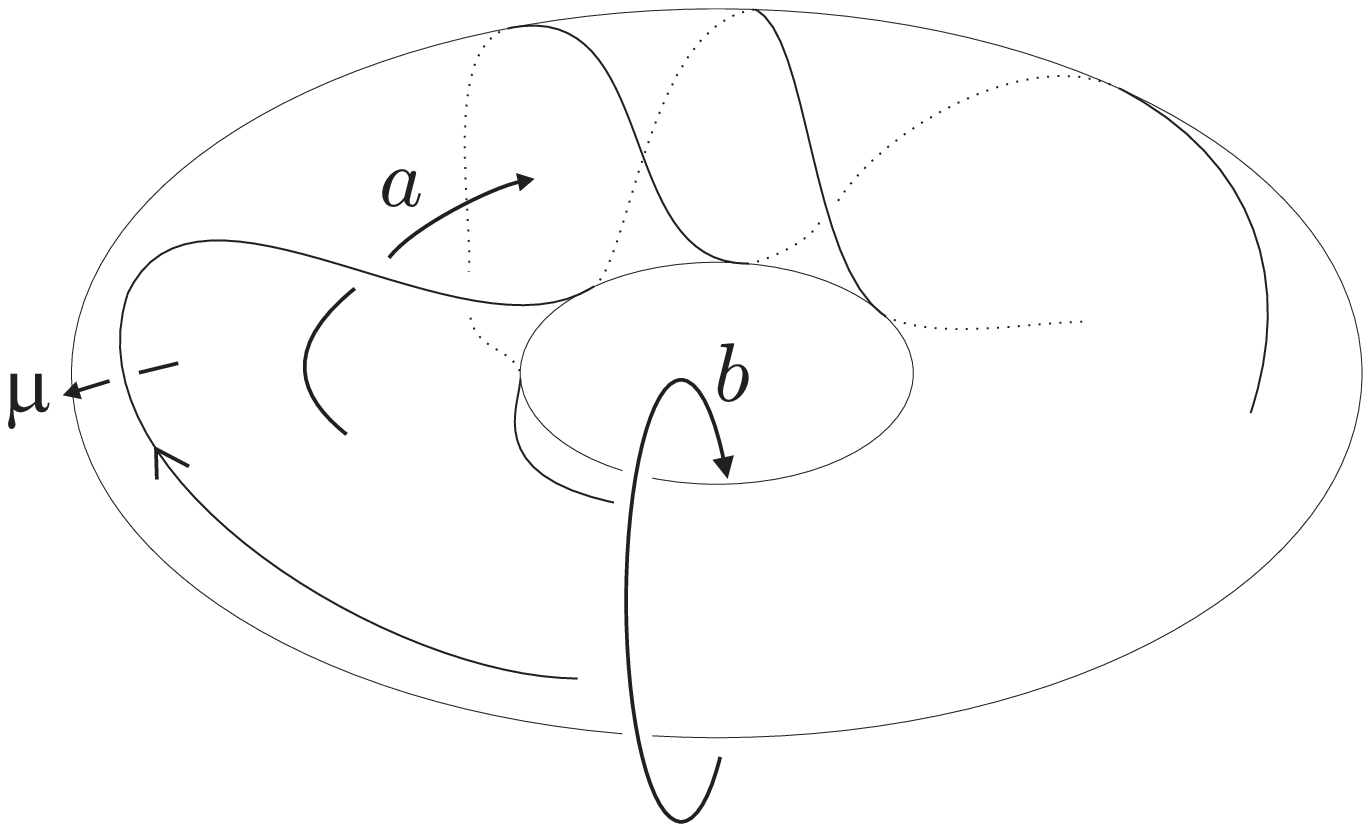}
\caption{}\label{fig:torusknot}
\end{figure}

It is well known that $M_2$ admits two Seifert fibrations.
One is a fibration over the disk with two exceptional fibers of index $2$, and the other 
is that over the M\"{o}bius band with no exceptional fiber.
Then we can choose the generators $\{x,y\}$ of $\pi_1(M_2)$ so that
$x$ corresponds to an exceptional fiber in the first fibration, and 
$y$ corresponds to a regular fiber in the second fibration.
Thus we obtain that $\pi_1(M_2)=\langle x, y : x^{-1}yx=y^{-1}\rangle$, and that
$\pi_1(\partial M_2)$ is generated by $x^2$ and $y$.

To get a presentation of $\pi_1(M)$, we have to examine the identification between $\partial M_1$
and $\partial M_2$.

\begin{lemma}
Under the identification between $\partial M_1$ and $\partial M_2$, 
$\mu$ and $h$ on $\partial M_1$ correspond to $y^{-1}$ and $y^{-1}x^2$ on $\partial M_2$, respectively.
\end{lemma}

\begin{proof}
Consider two loops $z$ and $w$ on the boundary of 
the Montesinos tangle $M(-1/2,1/2)$ as illustrated in Figure \ref{fig:tangle}.
Then $z$ and $w$ lift to two copies of $x^2$ and $y$, respectively
(see \cite[Chapter 12]{BZ}).

\begin{figure}[htbp]
\includegraphics*[scale=0.5]{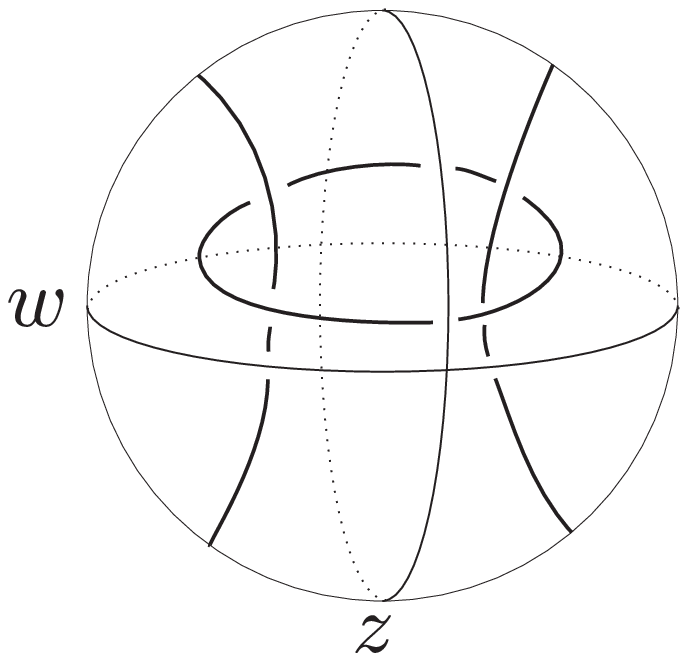}
\caption{}\label{fig:tangle}
\end{figure}

Next, we replace the Montesinos tangle $M(-1/2,1/2)$ with the $0$-tangle as shown in Figure \ref{fig:identify}, where we insert a narrow band $f$ to chase $z$.
The result is a trivial knot.
By taking the double branched cover along this trivial knot,
the band $f$ lifts to a knotted annulus, whose core forms the torus knot of type $(2,2m+1)$.
(This proves that the double branched cover of $M(-1/2,-m/(2m+1))$ is the exterior of
the torus knot of type $(2,2m+1)$.)
Also, the framing determined by $f$ has slope $(4m+1)/1$.
Recall that $h$ has slope $(4m+2)/1$.
Hence we can choose the orientations of $x$ and $y$ so that 
$y^{-1}$ and $x^2$ correspond to the meridian $\mu$ and
$\mu^{-1}h$, respectively.
\end{proof}

\begin{figure}[htbp]
\includegraphics*[scale=0.7]{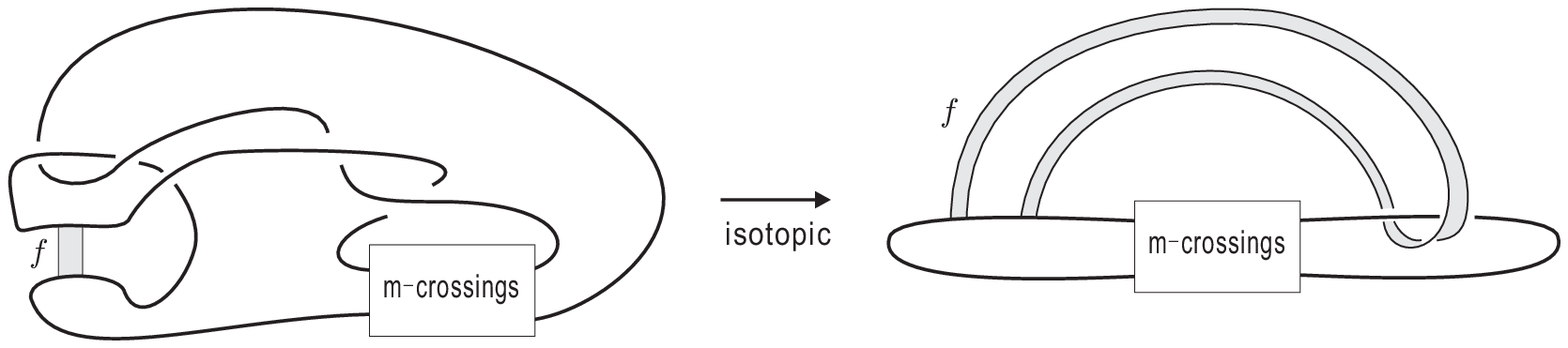}
\caption{}\label{fig:identify}
\end{figure}


Thus we have shown the following.

\begin{proposition}\label{prop:pi1}
Let $K$ be the $m$-twist knot with $m\ne 0,-1$.
Then $4$-surgery on $K$ yields a graph manifold $M$
which is the union of the twisted $I$-bundle over the Klein bottle and the knot
exterior of torus knot of type $(2,2m+1)$.
Furthermore, its fundamental group has a presentation
\[
\pi_1(M)=\langle a,b,x,y : a^2=b^{2m+1}, x^{-1}yx=y^{-1}, \mu=y^{-1}, h=y^{-1}x^2\rangle,
\]
where $\mu=b^{-m}a$ and $h$ correspond to a meridian and
a regular fiber of the torus knot exterior \textup{(}with the Seifert fibration\textup{)}, respectively. 
\end{proposition}

\begin{remark}
Our presentation in Proposition \ref{prop:pi1} is equivalent to that of \cite{CLW}
for the case $m=1$.
\end{remark}

\section{Normal families of left-orderings}

Let $G$ be a left-orderable non-trivial group.
This means that $G$ admits a strict total ordering $<$ such that
$a<b$ implies $ga<gb$ for any $g\in G$. 
This is equivalent to the existence
of a \textit{positive cone\/} $P \ (\ne \varnothing)$, which is a semigroup and
gives a disjoint decomposition
$P\sqcup \{1\} \sqcup P^{-1}$.
For a given left-ordering $<$, the set
$P=\{g\in G\mid g>1\}$ gives a positive cone.
Any element of $P$ (resp.~$P^{-1}$) is said to be \textit{positive\/} (resp.~\textit{negative\/}).
Conversely, given a positive cone $P$,
declare $a<b$ if and only if $a^{-1}b\in P$.
This defines a left-ordering.

We denote by $\mathrm{LO}(G)$ the set of all positive cones in $G$.
This is regarded as the set of all left-orderings of $G$ as mentioned above.
For $g\in G$ and $P\in \mathrm{LO}(G)$, let $g(P)=g^{-1}Pg$.
This gives a $G$-action on $\mathrm{LO}(G)$.
In other words,
for a left-ordering $<$ of $G$,
an element $g$ sends $<$ to a new left-ordering $<^g$ defined as follows:
$a<^g b$ if and only if $ag<bg$.
We say that $<$ and $<^g$ are \textit{conjugate\/} orderings.
Also, a family $L\subset \mathrm{LO}(G)$ is said to be \textit{normal\/} if it is $G$-invariant.

\begin{example}\label{ex:kb}
Let $G=\langle x,y : x^{-1}yx=y^{-1}\rangle$.
This is the fundamental group of the Klein bottle.
It is known that $G$ admits exactly four left-orderings.
We will define two normal families $L_{+}$ and $L_{-}$ of left-orderings as follows.
Consider a short exact sequence
\[
1\to \langle y \rangle \to G \to \langle x\rangle \overset{q}{\to} 1.
\]

For $g\in G$, define $1<_{++} g$ if $q(g)=x^s$ with $s>0$, or
$q(g)=1$ and $g=y^r$ with $r>0$.
Similarly, 
define $1<_{+-} g$ if $q(g)=x^s$ with $s>0$, or
$q(g)=1$ and $g=y^r$ with $r<0$.
Then we can easily prove that 
$L_{+}=\{<_{++},<_{+-}\}$ gives a normal family of $\mathrm{LO}(G)$.

Similarly, 
define $1<_{-+} g$ if $q(g)=x^s$ with $s<0$, or
$q(g)=1$ and $g=y^r$ with $r>0$.
And,
define $1<_{--} g$ if $q(g)=x^s$ with $s<0$, or
$q(g)=1$ and $g=y^r$ with $r<0$.
Then $L_{-}=\{<_{-+},<_{--}\}$ gives another normal family.
\end{example}

We need one more notion.
For $i=1,2$,
let $G_i$ be a left-orderable group and $H_i$ a subgroup of $G_i$, and
let $L_i\subset \mathrm{LO}(G_i)$ be a family of left-orderings.
Let $\phi:H_1\to H_2$ be an isomorphism.
We call that $\phi$ is \textit{compatible\/} for the pair $(L_1,L_2)$ if
for any $P_1\in L_1$, there exists $P_2\in L_2$ such that
$h_1\in P_1$ implies $\phi(h_1)\in P_2$ for any $h_1\in H_1$.

\begin{theorem}[Bludov-Grass \cite{BG}]\label{thm:BG}
For $i=1,2$,
let $G_i$ be a left-orderable group and $H_i$ a subgroup of $G_i$.
Let $\phi:H_1\to H_2$ be an isomorphism.
Then the free product with amalgamation $G_1*G_2\ (H_1\overset{\phi}{\cong} H_2)$
is left-orderable if and only if there exist normal families $L_i\subset \mathrm{LO}(G_i)$
for $i=1,2$ such that $\phi$ is compatible for $(L_1,L_2)$.
\end{theorem}

\section{An ordering of torus knot group}

For $n\ge 1$, let $\Gamma_n=\langle b,c: b=cb^{n}c\rangle$.
Navas \cite{Na} proved that the semigroup generated by $\{b,c\}$
gives a positive cone, so, a left-ordering of $\Gamma_n$.
In other words, 
an element $w\in \Gamma_n$ is positive (resp.~negative) if
$w$ can be written in only positive  (resp.~negative) powers of $b, c$.

Let $\Delta=b^{n+1}$. Then $\Delta>1$.
It is easy to see that $\Delta$ is central. 
(In fact, $\Delta$ generates the center of $\Gamma_n$.)
Also, $b^{-1}=b^n\Delta^{-1}$ and $c^{-1}=b^ncb^n\Delta^{-1}$.
Thus as Navas observes,
every element $w\in \Gamma_n$ can be written in a form $u\Delta^\ell$
for some trivial or positive $u$ and $\ell\in \mathbb{Z}$.

Furthermore, he shows that every element $w\in \Gamma_n$ has a \textit{normal form}
\[
w=c^{n_0}b^{m_1}c^{n_1}\cdots c^{n_k-1}b^{m_k}c^{n_k}\Delta^\ell=u\Delta^\ell
,\]
with the properties
\begin{itemize}
\item[(i)] $n_i>0$ for $0<i<k$, $n_0\ge 0$, $n_k\ge 0$;
\item[(ii)] $m_i\in \{1,2,\dots,n-1\}$ for $1<i<k$;
\item[(iii)] $m_1\in \{1,2,\dots,n-1\}$ (resp.~$\{1,2,\dots,n\}$) if
$n_0>0$ (resp.~$n_0=0$); similarly, $m_k\in \{1,2,\dots,n-1\}$ (resp.~$\{1,2,\dots,n\}$)
if $n_k>0$ (resp.~$n_k=0$);
\end{itemize}
and $\ell\in \mathbb{Z}$.

The next is proved in section 2 of \cite{Na}.

\begin{lemma}[\cite{Na}]\label{lem:navas}
Let $w=u\Delta^\ell$ be a normal form of a non-trivial element $w\in \Gamma_n$.
\begin{itemize}
\item[(i)] If $u=1$, then $w$ is positive or negative, according to the sign of $\ell$.
\item[(ii)] If $u\ne 1$ and $\ell\ge 0$, then $w$ is positive.
If $u\ne 1$ and $\ell<0$, then $w$ is negative.
\end{itemize}
\end{lemma}

\begin{lemma}\label{lem:cofinal}
For any $w\in \Gamma_n$,
there exists an integer $\ell$ such that
$\Delta^{\ell}<w<\Delta^{\ell+1}$.
\end{lemma}

\begin{proof}
Let $w=u\Delta^{\ell}$ be a normal form, where $u$ is trivial or positive.
If $u=1$, then the conclusion is clear.
So, let $u>1$.
Recall that $\Delta$ is central.
Then $\Delta^{\ell}<\Delta^{\ell}u$, and $\Delta^{\ell}u<\Delta^{\ell+1}$ by Lemma \ref{lem:navas}(ii).
Thus we have $\Delta^{\ell}<w<\Delta^{\ell+1}$.
\end{proof}

Let $G_{2m+1}=\langle a,b  : a^2=b^{2m+1} \rangle$.
We are interested in the case where $|m|\ge 2$.
This is isomorphic to the knot group of the torus knot of type $(2,2m+1)$.
It is well known (see \cite{BZ}) that $h=a^2=b^{2m+1}$ is a central element, which
corresponds to a regular fiber of the torus knot exterior with 
a (unique) Seifert fibration, and the meridian $\mu$ is $b^{-m}a$.

Suppose $m>0$. Then
\begin{eqnarray*}
G_{2m+1} &=&\langle a,b,c : a^2=b^{2m+1}, c=ba^{-1}\rangle \\
    &=&\langle b,c : b=cb^{2m}c\rangle.
\end{eqnarray*}

Thus this is $\Gamma_{2m}$ in Navas's notation.
We remark that $\Delta=h$.

Assume $m<0$. Set $n=-m-1\ (\ge 1)$.
Then
\begin{eqnarray*}
G_{2m+1} &=&\langle a,b,c : a^2=b^{-2n-1}, c=ab\rangle \\
    &=&\langle b,c : b=cb^{2n}c\rangle.
\end{eqnarray*}

Hence $G_{2m+1}=\Gamma_{2n}=\Gamma_{-2m-2}$.
We should remark that $\Delta=b^{2n+1}=b^{-2m-1}=h^{-1}$.

In either case, we can introduce Navas's left-ordering to $G_{2m+1}$, denoted by $<$, hereafter.

\begin{lemma}\label{lem:meridian}
Both $\mu$ and $h$ are  either positive or negative, according to the sign of $m$.
\end{lemma}

\begin{proof}
Assume $m>0$.
Since $b^{-m}a=b^{m+1}a^{-1}$ and $c=ba^{-1}$,
$\mu=b^{m+1}a^{-1}=b^m(ba^{-1})=b^mc$.
Thus $\mu$ is positive by Lemma \ref{lem:navas}.
Also, $h=\Delta>1$.

Assume $m<0$, and set $n=-m-1$ as before.
Then $\mu=b^{n+1}a=b^{n+1}cb^{-1}=b^{n+1}cb^{2n}\Delta^{-1}$, since
$a=cb^{-1}$ and $b^{-1}=b^{2n}\Delta^{-1}$.
Hence $\mu<1$.
Finally, $h=\Delta^{-1}<1$.
\end{proof}

\begin{lemma}\label{lem:mh}
For any integer $r$, $\mu^r<\Delta$.
\end{lemma}

\begin{proof}
Suppose $m>0$.
As in the proof of Lemma \ref{lem:meridian}, $\mu=b^mc$.
If $r>0$, then
$\Delta^{-1}\mu^r=\mu^r\Delta^{-1}=(b^mc)^r\Delta^{-1}<1$ by the criterion of Lemma \ref{lem:navas}.
Thus $\mu^r<\Delta$.

If $r<0$, then set $k=-r$.
Then $\mu^r=\mu^{-k}=(c^{-1}b^{-m})^k$.
By $c^{-1}=b^{2m}cb^{2m}\Delta^{-1}$,
we have $\mu^r=(b^{2m}cb^m)^k\Delta^{-k}$.
When $k=1$, this is a normal form, so $\mu^r<1<\Delta$.
When $k>1$, $\mu^r=(b^{2m}cb^m)(b^{2m}cb^m)\cdots (b^{2m}cb^m)\Delta^{-k}
=b^{2m}cb^{m-1}\cdots cb^m\Delta^{-1}$.
Again, $\mu^r<1<\Delta$.

Now, assume $m<0$.
Set $n=-m-1$ as before.
As in the proof of Lemma \ref{lem:meridian},
$\mu=b^{n+1}cb^{-1}=b^{n+1}cb^{2n}\Delta^{-1}$.
If $r=1$, then
$\mu=b^{n+1}cb^{2n}\Delta^{-1}<1<\Delta$.
If $r>1$, then $\mu^r=(b^{n+1}cb^{-1})^r=b^{n+1}cb^n\cdots b^ncb^{-1}=b^{n+1}cb^n\cdots b^n c b^{2n}\Delta^{-1}<1<\Delta$.
If $r<0$, set $k=-r$.
Then $\mu^r=\mu^{-k}=(bc^{-1}b^{-n-1})^k=(cb^{n-1})^k<\Delta$.
\end{proof}

The next lemma is proved by a similar argument to
that of \cite{CLW}.

\begin{lemma}\label{lem:interval}
For any element $g\in G_{2m+1}$ and an integer $r$,
$\Delta^{-1}<g^{-1}\mu^r g<\Delta$. 
\end{lemma}

\begin{proof}
For a given $g$,
there exists an integer $\ell$ such that $\Delta^{\ell}<g<\Delta^{\ell+1}$
by Lemma \ref{lem:cofinal}.
Since $\Delta$ is central, we also have $\Delta^{-\ell-1}<g^{-1}<\Delta^{-\ell}$.

Here, assume that $\Delta<g^{-1}\mu^r g$ for contradiction.
Then $\Delta=\Delta^{\ell}\Delta^2\Delta^{-\ell-1}<g(g^{-1}\mu^{r}g)^2g^{-1}=\mu^{2r}$.
This contradicts Lemma \ref{lem:mh}.

Assume $g^{-1}\mu^r g<\Delta^{-1}$.
Then
$\mu^{2r}=g(g^{-1}\mu^r g)^2g^{-1}<\Delta^{\ell+1}\Delta^{-2}\Delta^{-\ell}=\Delta^{-1}$.
So, $\Delta<\mu^{-2r}$, which contradicts Lemma \ref{lem:mh} again.
\end{proof}

Unfortunately, Navas's ordering does not satisfy the so-called Property S,
but we have a weaker result, which is sufficient to our purpose.

\begin{lemma}\label{lem:propS}
For any conjugate ordering $<^g$ of Navas's ordering $<$ of $G_{2m+1}$,
assume $1<^g\mu^rh^s$. Then we have the following.
\begin{itemize}
\item[(1)] If $m>0$, then
\begin{itemize}
\item[(i)] $s>0$\textup{;} or
\item[(ii)] $s=0$ and 
$r>0$ \textup{(}resp.~$r<0$\textup{)} if $g^{-1}\mu g>1$ 
\textup{(}resp.~$g^{-1}\mu g<1$\textup{)}.
\end{itemize}
\item[(2)] If $m<0$, then
\begin{itemize}
\item[(i)] $s<0$\textup{;} or
\item[(ii)] $s=0$ and 
$r>0$ \textup{(}resp.~$r<0$\textup{)} if $g^{-1}\mu g>1$ 
\textup{(}resp.~$g^{-1}\mu g<1$\textup{)}.
\end{itemize}
\end{itemize}
\end{lemma}

\begin{proof}
By definition, $1<g^{-1}\mu^rh^sg$.
Then $g^{-1}\mu^{-r}g<h^s$.

(1) Assume $m>0$.  Then $\Delta=h$.
By Lemma \ref{lem:interval}, we have $s\ge 0$.
So, suppose $s=0$.
If $g^{-1}\mu g>1$, then $g^{-1}\mu^{-r}g<1$ if and only if $r>0$. 
Similarly, 
if $g^{-1}\mu g<1$, then $g^{-1}\mu^{-r}g<1$ if and only if $r<0$.

(2) Assume $m<0$. Then $\Delta=h^{-1}$.
By Lemma \ref{lem:interval}, we have $s\le 0$.
When $s=0$, the argument is the same as above.
\end{proof}

\section{Proofs}

\begin{proof}[Proof of Theorem \ref{thm:main}]
Let $M$ be the resulting manifold by $4$-surgery on the $m$-twist knot.
Let $M_1$ be the exterior of the torus knot of type $(2,2m+1)$
and $M_2$ be the twisted $I$-bundle over the Klein bottle.
Also, let $G_i=\pi_1(M_i)$ and $H_i=\pi_1(\partial M_i)$.
Then by Proposition \ref{prop:pi1},
$\pi_1(M)$ is the free product with amalgamation $G_1*G_2\ (H_1\overset{\phi}{\cong}H_2)$
where $G_1=\langle a,b : a^2=b^{2m+1}\rangle$, $G_2=\langle x,y : x^{-1}yx=y^{-1}\rangle$,
and $\phi(\mu)=y^{-1}$, $\phi(h)=y^{-1}x^2$.

For $\mathrm{LO}(G_1)$, let $L_1$ be the (normal) family of all conjugate orderings
of Navas's ordering.
For $\mathrm{LO}(G_2)$, set $L_2$ to be the normal family $L_{+}=\{<_{++},<_{+-}\}$ or 
$L_{-}=\{<_{-+},<_{--}\}$, defined in Example \ref{ex:kb}, according to the sign of $m$.
To show that $\pi_1(M)$ is left-orderable, it is sufficient to verify that
$\phi$ is compatible for the pair $(L_1,L_2)$ by Theorem \ref{thm:BG}.

Let $<^g\in L_1$.
Suppose $1<^g \mu^rh^s$.
Assume $m>0$.
According as  $g^{-1}\mu g$ is positive or negative with respect to Navas's ordering,  
we choose $<_{+-}$ or $<_{++}$ from $L_2$, respectively.
Since $\phi(\mu^rh^s)=y^{-r}(y^{-1}x^2)^s$,
$q(\phi(\mu^rh^s))=x^{2s}$.
Then $\phi(\mu^rh^s)$ is positive by Lemma \ref{lem:propS}.
When $m<0$, we choose $<_{--}$ or $<_{-+}$ from $L_2$.
\end{proof}

\begin{proof}[Proof of Corollary \ref{cor}]
Let $K$ be the $m$-twist knot.
Then it is sufficient to consider the case where $|m|\ge 2$, because
the conclusion for the figure-eight knot is settled by \cite{BGW, CLW}.
According to the classification of exceptional Dehn surgery on $2$-bridge knots \cite{BW},
$K$ admits exactly five exceptional (non-trivial) Dehn surgeries.
More precisely, those slopes are $0,1,2,3$ and $4$.
For $r=1,2$ or $3$, $r$-surgery yields a small Seifert fibered manifold (\cite{BW}).
Since $K$ is not fibered, it does not admit an $L$-space surgery by \cite{N}. Hence
such a Seifert fibered manifold has left-orderable fundamental group by \cite[Theorem 4]{BGW}.
For $r=0$, the resulting manifold has positive betti number, so
its fundamental group is left-orderable by \cite{BRW}.
Finally, our Theorem \ref{thm:main} solves the remaining case $r=4$.
\end{proof}

\bibliographystyle{amsplain}

\end{document}